\documentclass[10pt]{article}
\usepackage{amssymb, amsfonts, amsthm, mathrsfs, hyperref, lscape}
\usepackage{graphicx, xcolor}
\usepackage[all]{xy}
\usepackage[leqno]{amsmath}

\theoremstyle{definition}
\newtheorem{thm}{Theorem}[section]
\newtheorem{lemma}[thm]{Lemma}
\newtheorem{cor}[thm]{Corollary}
\newtheorem{prop}[thm]{Proposition}

\newtheorem{conj}[thm]{Conjecture}

\theoremstyle{remark}
\newtheorem{rmk}[thm]{Remark}



\def\red{\textcolor{red}}
\def\zz{\mathbb{Z}}
\def\qq{\mathbb{Q}}
\def\ud{\mathrm{d}}
\def\ccc{\mathbb{C}}
\def\coker{\mathrm{coker}~}

\title{On Lower Central Series Quotients of Finitely Generated Algebras over $\zz$}
\author{Katherine Cordwell\footnote{University of Maryland}, Teng Fei\footnote{Massachusetts Institute of Techonology}~ and Kathleen Zhou\footnote{Harvard University}}
\date{}

\begin{document}
\maketitle{}

\begin{abstract}
Let $A$ be an associative unital algebra, $B_k$ its successive quotients of lower central series and $N_k$ the successive quotients of ideals generated by lower central series. The geometric and algebraic aspects of $B_k$'s and $N_k$'s have been of great interest since the pioneering work of \cite{feigin2007}. In this paper, we will concentrate on the case  where $A$ is a noncommutative polynomial algebra over $\zz$ modulo a single homogeneous relation. Both the torsion part and the free part of $B_k$'s and $N_k$'s are explored. Many examples are demonstrated in detail, and several general theorems are proved. Finally we end up with an appendix about the torsion subgroups of $N_k(A_n(\zz))$ and some open problems.
\end{abstract}

\section{Introduction}

Let $A$ be an associative unital algebra over a commutative ring $R$. There is a natural Lie algebra structure on $A$, whose Lie bracket is given by $[a,b]=a\cdot b-b\cdot a$. Let $L_k(A)$ be the lower central series of $A$. They are defined inductively by $L_1(A)=A$ and $L_{k+1}(A)=[A,L_k(A)]$. Denote by $M_i(A)$ the two-sided ideal generated by $L_k(A)$ in $A$. An easy computation shows (see \cite{kerchev2013} for instance) $M_i(A)=A\cdot L_i(A)$. Finally, we define $B_k(A)=L_k(A)/L_{k+1}(A)$ and $N_k(A)=M_k(A)/M_{k+1}(A)$ as successive quotients.

The study of $B_i$'s began in \cite{feigin2007}, where the interest was on the case where $R=\qq$ and $A=A_n(\qq)$ is the free associative algebra with $n$ generators. In their paper, Feigin and Shoikhet observed that each $B_k(A_n(\qq))$ for $k\geq 2$ admits a $W_n$-module structure, where $W_n$ is the Lie algebra of polynomial vector fields on $\qq^n$. In particular, they showed that $B_2(A_n(\qq))$ is isomorphic, as a graded vector space, to $\Omega_{\textrm{closed}}^{\textrm{even}>0}(\qq^n)$, the space of closed polynomial forms on $\qq^n$ of positive even degree.

Later on, Dobrovolska et al. \cite{dobrovolska2008} extended Feigin and Shoikhet's results to more general algebras over $\ccc$, and Balagovic and Balasubramanian \cite{balagovic2011} explored the case $A=A_n(\ccc)/(f)$ where $f$ is a generic homogeneous polynomial. On the other hand, Bhupatiraju et al. \cite{bhupatiraju2012} studied the case where $A$ is the free algebra over $\zz$ or finite fields. Specifically, they were able to describe almost all the torsion appearing in $B_2(A_n(\zz))$.

The story of the $N_i$'s is comparably briefer. The only literature on the $N_i$'s as far as the authors know is \cite{etingof2009} and \cite{kerchev2013}, where the Jordan-H\"older series of $N_i$'s (as $W_n$-modules) are investigated for $A=A_n(\qq)$.

The goal of this paper is to understand $B_k$'s and $N_k$'s for $A=A_n(\zz)/(f)$, a $\zz$-algebra generated by $n$ elements with a single homogeneous relation. We are especially interested in which torsion groups show up.

The organization of our paper is as follows. In Section 2, we begin with an easy example, the $q$-polynomials, where everything can be computed explicitly. From Section 3 to Section 5, many patterns of $B_k$ and $N_k$ for general $A$ are observed, and some of them are proven. Finally, we end up with an appendix on torsion subgroups of $N_k(A_n(\zz))$.

\section{An Example: The $q$-Polynomials}

Let $A=\zz\langle x,y\rangle/(yx-qxy)$, $q\in\zz$, be the algebra of $q$-polynomials. The case $q=1$ is trivial, because $A$ is commutative and therefore all the $B_k$'s and $N_k$'s are 0 for $i\geq 2$. Let us assume that $q\neq\pm1$ at this moment, and the goal is to compute $B_k$'s and $N_k$'s explicitly.

Let us assign $x$ and $y$ to be of degree $(1,0)$ and $(0,1)$ respectively in $\zz\langle x,y\rangle$. Clearly this assignment descends to $A$, therefore we get a $\zz^2$-grading on $A$ as well as the various quotients $B_k$'s and $N_k$'s. We will use $L_k[i,j]$ (resp. $M_k[i,j]$, $B_k[i,j]$, $N_k[i,j]$) to denote the degree $(i,j)$ part of $L_k$ (resp. $M_k$, $B_k$, $N_k$).

Using the relation $yx=qxy$ repeatedly, we know that each $L_k[i,j]$ is a free abelian group of rank at most 1, whose basis is of the form $S^k_{i,j}x^iy^j$ for some $S^k_{i,j}\in\zz$. Similarly, we write the basis of $M_k[i,j]$ as $T^k_{i,j}x^iy^j$. Our goal is to compute $S^k_{i,j}$ and $T^k_{i,j}$.

Clearly when $k>i+j$, we have $S^k_{i,j}=T^k_{i,j}=0$. So let us first look at the case $k=i+j$. By definition of lower central series, we have \[L_k[i,j]=[x,L_{k-1}[i-1,j]]+[y,L_{k-1}[i,j-1]],\] thus we obtain the recursive formula \[S^k_{i,j}=\gcd\left(S^{k-1}_{i-1,j}(q^j-1),S^{k-1}_{i,j-1}(q^i-1)\right).\] Using some elementary number theory, by induction we conclude for $k=i+j$, \[S^k_{i,j}=T^k_{i,j}=(q-1)^{i+j-2}\cdot(q^{\gcd(i,j)}-1).\] Similarly, we can prove that for $k<i+j$, \[S^k_{i,j}=(q-1)^{k-2}\cdot(q^{\gcd(i,j)}-1),~~~~ T^k_{i,j}=(q-1)^{k-1}.\]

As a result, we can formulate the following proposition:
\begin{prop}
For the $q$-polynomial algebra $A=\zz\langle x,y\rangle/(yx-qxy)$ with $q\neq\pm1$, we have \[\begin{split}&B_k(A)[i,j]\cong\begin{cases}\zz_{|q-1|} & i,j>0, k<i+j\\ \zz & i,j>0, k=i+j\\ 0&\textrm{elsewhere}\end{cases},\\ &N_k(A)[i,j]\cong\begin{cases}\zz_{|q-1|} & i,j>0, k<i+j-1 \\ \zz_{|q^{\gcd(i,j)}-1|} & i,j>0, k=i+j-1\\ \zz & i,j>0, k=i+j\\ 0&\textrm{elsewhere}\end{cases}\end{split}\]
\end{prop}

As a corollary, we have the following
\begin{cor}
All primes except those dividing $q$ appear in the torsion subgroup of $N_k(A)$.
\end{cor}
\begin{proof}
Fermat's little theorem.
\end{proof}
The case $q=-1$ can be worked out in like manner, so we omit the proof here.

\begin{prop}
For $A=\zz\langle x,y\rangle/(yx+xy)$, we have \[\begin{split} &B_k(A)[i,j]=\begin{cases} \zz_2 & i+j>k,~i,j>0,\textrm{ not all even}\\ \zz & i+j=k,~i,j>0,\textrm{ not all even}\\ 0 & \textrm{elsewhere}\end{cases}\\ &N_k(A)[i,j]=\begin{cases} \zz_2 & \begin{cases}i+j>k,~i,j>0,\textrm{ not all even}\\ i+j>k+1,~i,j>0\textrm{ even}\end{cases}\\ \zz & \begin{cases}i+j=k,~i,j>0,\textrm{ not all even}\\ i+j=k+1,~i,j>0\textrm{ even}\end{cases}\\ 0 & \textrm{elsewhere}\end{cases}\end{split}\]
\end{prop}

\section{A Basis for $N_2(\qq\langle x,y\rangle/(x^m+y^m))$}

Another simple example to study is the case $A(m)=\zz\langle x,y\rangle/(x^m+y^m)$ for $m\geq2$ an integer. If we assign both $x$ and $y$ to be of degree 1, then we get a graded algebra structure on $A$. Again, let $L_k[d]$ (resp. $M_k[d]$, $B_k[d]$, $N_k[d]$) be the degree $d$ part of $L_k$ (resp. $M_k$, $B_k$, $N_k$). Thanks to \cite{bosma1997}, Magma helps us to get the data presented in Table 1:

\begin{table}[h]
\centering
\begin{tabular}{l|ccccccc}
\hline
r$(N_2[d])$&2&3&4&5&6&7&8\\
\hline
2&1\\
3&1&2&1\\
4&1&2&3&2&1\\
5&1&2&3&4&3&2&1\\
\hline
\end{tabular}
\caption{The ranks of $N_2(A(m))[d]$ for $m=2,3,4,5$.}
\end{table}

Here each entry represents the rank of $N_2(A(m))[d]$, where the first row indicates the degree $d$ and the left column values of $m$. The pattern above is extremely clear, which leads to the following proposition:
\begin{prop}\label{prop31}
For $A=\zz\langle x,y\rangle/(x^m+y^m)$, we have \[\textrm{Rank}(N_2[d])=\begin{cases}d-1 & d<m \\ 2m-d-1 & m\leq d\leq 2m-2 \\ 0 & \textrm{otherwise}\end{cases}\]
\end{prop}
Because we are only concerned about ranks, we may tensor with $\qq$. Then Proposition \ref{prop31} follows directly from finding a basis for $N_2(\qq\langle x,y\rangle/(x^m+y^m))$, as Balagovic and Balasubramanian did for $B_2(\qq\langle x,y\rangle/(x^m+y^m))$ in \cite{balagovic2011}, Section 3.

To proceed, let us introduce a new variable $u=[x,y]\in A$. By abuse of notations, we are going to use letters $x,y$ or $u$ to represent corresponding classes in the quotient algebras.
\begin{lemma}\label{lemma32}
In $A/M_3(A)$, we have \[u^2=[u,x]=[u,y]=x^{m-1}u=y^{m-1}u=0.\]
\end{lemma}
\begin{proof}
Direct calculation.
\end{proof}
\begin{rmk}
The triple $\{x,y,u\}$ in $A/M_3(A)$ satisfies the commutative relations of the standard generators of a Heisenberg algebra.
\end{rmk}
The above lemma tells us that $x^iy^j~(0\leq i<m)$ and $x^iy^ju~(i,j<m-1)$ span $A/M_3(A)$. Let us consider the short exact sequence \[0\to N_2(A)\to A/M_3(A)\xrightarrow{\pi}A/M_2(A)\to0.\] Note $A/M_2(A)\cong\qq[x,y]/(x^m+y^m)$ is the abelianization of $A$, we easily see that the images of $x^iy^j~(0\leq i<m)$ under $\pi$ are linearly independent. As $\pi(u)=0$, a natural guess would be the following:
\begin{prop}
$x^iy^ju~(i,j<m-1)$ form a basis for $N_2(A)$.
\end{prop}
\begin{proof}
We only need to prove that they are linearly independent in $A/M_3(A)$. Rewrite $A/M_3(A)$ as $(A_2/M_3(A_2))/(x^m+y^m)$, where $(x^m+y^m)$ is the ideal generated by $(x^m+y^m)$ in $A_2/M_3(A_2)$, we claim that any nonzero linear combination of $x^iy^ju$ does not lie in $(x^m+y^m)$.

Recall that \cite{feigin2007} Lemma 2.1.2 showed there is an isomorphism of algebras $\varphi:A_2/M_3(A_2)\to\Omega^{\textrm{even}}(\qq^2)_*$, where $\Omega^{\textrm{even}}(\qq^2)_*$ is the algebra of even polynomial differential forms on $\qq^2$ with the twisted product $\alpha*\beta=\alpha\wedge\beta+(-1)^{\deg\alpha}\ud\alpha\wedge\ud\beta$. To be precise, $\varphi$ is defined by $\varphi(x)=x$, $\varphi(y)=y$.

Now everything can be translated into differential forms, and the verification of our claim in that context is straightforward.
\end{proof}
As a consequence, Proposition \ref{prop31} follows from dimensional counting of the basis $x^iy^ju~(i,j<m-1)$.
\begin{rmk}
If we work more carefully, we can figure out the torsion subgroups in $N_2(A)$ using the same method.
\end{rmk}
Furthermore, after taking a look at the data table of higher $N_k(A)$, say $k=3$, one may easily formulate the following conjecture:
\begin{conj}
For $A=\zz\langle x,y\rangle/(x^m+y^m)$, we have \[\textrm{Rank}(N_3(A)[d])=\begin{cases}3d-7 & 3\leq d\leq m+1 \\ 6m-3d+1 & m+2\leq d\leq 2m \\ 0 & \textrm{elsewhere}\end{cases}\]
\end{conj}

Actually, we can compute the rank of $N_2[d]$ explicitly for $A=\zz\langle x,y\rangle/(f)$, where $f$ is an arbitrary homogeneous relation of degree $m$. Let $f_{\mathrm{ab}}$ be the abelianization of $f$, and denote $\partial f_{\mathrm{ab}}/\partial x$ and $\partial f_{\mathrm{ab}}/\partial y$ by $f_x$ and $f_y$ respectively. Instead of the relations in Lemma \ref{lemma32}, the following equations are satisfied in $A/M_3(A)$: \[u^2=[u,x]=[u,y]=f_xu=f_yu=0.\] As before, $\{x^iy^ju\}$ spans $N_2(A)$, and the only relations are \[f_xu=f_yu=0.\] The relation $f_{\mathrm{ab}}u=0$ is redundant because of the Euler identity \[m f_{\mathrm{ab}}=xf_x+yf_y.\] Therefore there is a degree-preserving bijection \[\{\textrm{basis of } \qq[x,y]/(f_x,f_y)\}\cdot u\longleftrightarrow \{\textrm{basis of }N_2(A)\}.\] In other words, we have \begin{equation}\label{*}\textrm{Rank}(N_2(A)[d])=\textrm{Rank}(\zz[x,y]/(f_x,f_y)[d-2]).\end{equation}

We may write $f_{ab}$ as product of linear polynomials of the form $\alpha x+\beta y$ (over $\mathbb{C}$). It is easy to check that \[(\alpha x+\beta y)^l||\gcd(f_x,f_y)~~\textrm{ if and only if }~~(\alpha x+\beta y)^{l+1}||f_{ab}.\] Now let $\{(\alpha_ix+\beta_iy)\}_{i=1}^s$ be the set of distinct linear factors of $f_{ab}$, and let $m_i$ be the multiplicity of $(\alpha_ix+\beta_iy)$, which satisfies\[\sum_{i=1}^sm_i=m.\] The least common multiple of $f_x$ and $f_y$ is given by \[\frac{f_xf_y}{\prod_{i=1}^s(\alpha_ix+\beta_iy)^{m_i-1}},\] which is of degree $m+s-2$. Combining it with Eq. (\ref{*}), we conclude:
\begin{thm}\label{thm37}
Let $A=\zz\langle x,y\rangle/(f)$, where $f$ is a homogeneous polynomial of degree $m$, let $s$ be the number of distinct linear factors of $f_{ab}$ (over $\mathbb{C}$). Then \[\textrm{Rank}(N_2(A)[d])=\begin{cases}0 & d=0\\ d-1 & 1\leq d\leq m-1 \\2m-d-1 & m\leq d\leq m+s-1 \\ m-s & d\geq m+s\end{cases}\]
\end{thm}

\begin{rmk}
The computation here is consistent with the results we get in Sections 4 and 5.
\end{rmk}

\section{Rank Stabilization}

In this section, we are going to study the $B_k$'s and $N_k$'s associated with $A=\zz\langle x,y\rangle/(x^m)$. It is essentially different from what we have seen in Sections 1 and 2: If we look at the data produced by Magma (see Table 2), we find $\textrm{Rank}(B_k(A)[d])$ stabilizes at some positive integer for any $k$ when $d$ is large enough. On the other hand, the $B_k$'s of $\zz\langle x,y\rangle/(x^m+y^m)$ are finite dimensional. We will say more about it in the next section.

\begin{table}[h]
\centering
\begin{tabular}{l|cccccccc}
\hline
r$(B_k[d])$ & 2 & 3 & 4 & 5 & 6 & 7 & 8 & 9\\
\hline
$B_2$ & 1 & 2 & 2 & 2 & 2 & 2 & 2 & 2\\
$B_3$ & & 2 & 4 & 4 & 4 & 4 & 4 & 4\\
$B_4$ & & & 3 & 7 & 8 & 8 & 8 & 8\\
$B_5$ & & & & 6 & 13 & 16 & 16 & 16\\
\hline
\end{tabular}
\caption{The ranks of $B_k(A)[d]$ for $A=\zz\langle x,y\rangle/(x^3)$.}\label{table2}
\end{table}

To be more precise, we have the following theorem:
\begin{thm}\label{thm41}
Let $A=\zz\langle x,y\rangle/(x^m)$, for each $k\geq 2$, there exists some $l=l(k)\in\zz_{\geq0}$ such that $\textrm{Rank}(B_k[l])=\textrm{Rank}(B_k[j])$ for any $j\geq l$. Furthermore, if $k=2$, then $l\leq m$; if $k\geq 3$, then $l\leq 2k+m-5$.
\end{thm}
Again, we may work over $\qq$. To simplify our notation, $A_n$ means $A_n(\qq)$ in this section. The proof of our theorem relies heavily on the following results:
\begin{prop}(\cite{feigin2007})
For the free associative algebras $A=A_n$, $B_k(A_n)$ has a natural $W_n$-module structure for $k\geq2$. Here $W_n=\bigoplus_{i=1}^n\qq[x_1,\dots,x_n]\partial_i$ is the Lie algebra of polynomial vector fields over $\qq^n$.
\end{prop}
\begin{prop}(\cite{rudakov1974},\cite{dobrovolskae2008})\label{prop43}
Each $B_k(A_n)$ has a finite length Jordan-H\"older series with respect to the $W_n$-action. Moreover, all the irreducible subquotients of $B_k(A_n)$ is of the form $F_\lambda$, where $F_\lambda$ are certain tensor field modules parameterized by Young diagrams $\lambda=(\lambda_1,\dots,\lambda_n)$.
\end{prop}
For $n=2$, we have the following estimate on $|\lambda|$ for those $F_\lambda$ appearing in the Jordan-H\"older series of $B_k(A_2)$:
\begin{prop}(\cite{feigin2007}, \cite{arbesfeld2010})\label{prop44}
Let $k\geq3$, for $F_\lambda$ in the Jordan-H\"older series of $B_k(A_2)$, \[|\lambda|:=\lambda_1+\lambda_2\leq2k-3.\] For $k=2$, $B_2(A_2)\cong F_{(1,1)}$ as $W_2$-modules.
\end{prop}

Now let us prove Theorem \ref{thm41}.
\begin{proof}
Let $\qq[y]\partial_y=:W_1\subset W_2=\qq[x,y]\partial_x\oplus\qq[x,y]\partial_y$ be the Lie subalgebra of $W_2$. Clearly $B_k(A_2)$ has a $W_1$-module structure. Now we replace our associative algebra by $A=A_2/(x^m)$, the $W_2$-action does not exist any more, however, the $W_1$-action remains.

Claim: $B_i(A_2/(x^m))$ is of finite length as a $W_1$-module.

Let $\pi:B_k(A_2)\to B_k(A_2/(x^m))$ be the canonical map induced by the quotient $A_2\to A_2/(x^m)$, clearly, $\pi$ preserves the $W_1$-action. As $[W_1,x^mW_2]\subset x^mW_2$, $(x^mW_2)B_k(A_2)$ is a $W_1$-submodule of $B_k(A_2)$. Moreover, $\pi$ maps $(x^mW_2)B_k(A_2)$ to 0, and since $\pi$ is surjective, we conclude that $B_k(A_2/(x^m))$ is a quotient module of $B_k(A_2)/(x^mW_2)B_k(A_2)$. Therefore we only have to show that, as a $W_1$-module, $B_k(A_2)/(x^mW_2)B_k(A_2)$ is of finite length.

Recall that by Proposition \ref{prop43}, $B_k(A_2)$ is a finite-length $W_2$-module. Therefore there exists a sequence of $W_2$-modules \[0=M_n\subset M_{n-1}\subset\dots\subset M_1\subset M_0=B_k(A_2)\] such that each $P_j:=M_j/M_{j+1}$ is an irreducible $W_2$-module. Now we have the following commutative diagram: \[\xymatrix{0\ar[r] & (x^mW_2)M_{j+1}\ar[r]\ar[d] & (x^mW_2)M_j\ar[r]\ar[d] & (x^mW_2)M_j/(x^mW_2)M_{j+1}\ar[r]\ar[d]^f & 0 \\ 0\ar[r] & M_{j+1}\ar[r] & M_j\ar[r] & P_j\ar[r] & 0}\] where the rows are exact and the first two vertical arrows are inclusions. The snake lemma gives us the long exact sequence \[0\to\ker f\to M_{j+1}/(x^mW_2)M_{j+1}\to M_j/(x^mW_2)M_j\to P_j/(x^mW_2)P_j\to0.\] By induction on $j$, we only have to show that each $P_j/(x^mW_2)P_j$ is of finite length.

We know that each $P_j$ is of the form $F_{(p,q)}$ ($q\leq p$) and \[F_{(p,q)}=\qq[x,y]\otimes\mathrm{Sym}^{p-q}(\ud x,\ud y)\otimes(\ud x\wedge\ud y)^{\otimes q},\] where $\mathrm{Sym}(\ud x,\ud y)$ is the symmetric algebra generated by differentials $\ud x$ and $\ud y$. Moreover, $W_2$ acts on $F_{(p,q)}$ as Lie derivatives. To make the grading correct, $\ud x$ and $\ud y$ are each assigned to be of degree 1.

By direct computation, we see that $(x^mW_2)F_{(p,q)}=x^{m-1}F_{(p,q)}$, therefore \[F_{(p,q)}/(x^mW_2)F_{(p,q)}\cong\bigoplus_{l=0}^{m-2}\bigoplus_{j=q}^px^l(\ud x)^{p+q-j}F_j,\] where $F_j=\qq[y](\ud y)^j$ is an irreducible $W_1$-module. As a consequence, our claim is proven.

Furthermore, $F_j[d]$ is of dimension 1 as $d\geq j$, thus $x^l(\ud x)^{p+q-j}F_j[d]$ is of dimension 1 as $d\geq p+q+l$. Theorem \ref{thm41} follows from Proposition \ref{prop44} automatically.
\end{proof}

From Table \ref{table2} we cannot determine the Jordan-H\"older series (in terms of $W_1$-module) of $B_k(A_2/(x^m))$, because both $x$ and $\ud x$ carry degree 1, we cannot read off $j$ from the degree distribution of $x^l(\ud x)^{p+q-j}F_j$. However, if we compute ranks of $B_k(A_2/(x^m))$ refined by bidegree, we can deduce $j$ from data, thus determine the Jordan-H\"older series.

As a result, we get:
\begin{prop}
As $W_1$-modules, the Jordan-H\"older series of $B_k(A_2/(x^m))$ can be determined:
\[\begin{split}&B_2(A_2/(x^2))\sim F_1[1],\\ &B_3(A_2/(x^2))\sim F_1[2]+F_2[1],\\ &B_4(A_2/(x^2))\sim F_2[2]+F_3[1]+F_3[2],\\ &B_2(A_2/(x^3))\sim F_1[1]+F_1[2],\\ &B_3(A_2/(x^3))\sim F_1[2]+F_1[3]+F_2[1]+F_2[2],\\ &B_4(A_2/(x^3))\sim F_1[3]+F_2[2]+2F_2[3]+F_3[1]+2F_3[2]+F_3[3],\\ &B_2(A_2/(x^4))\sim F_1[1]+F_1[2]+F_1[3],\\ &B_3(A_2/(x^4))\sim F_1[2]+F_1[3]+F_1[4]+F_2[1]+F_2[2]+F_2[3],\\ &B_4(A_2/(x^4))\sim F_1[3]+F_1[4]+F_2[2]+2F_2[3]+2F_2[4]+F_3[1]+2F_3[2]+2F_3[3]+F_3[4],\\&......\end{split}\] Here, $F_j[d]$ denotes the $W_1$-module $F_j$ which begins with total degree $d$ in $x$ and $\ud x$.
\end{prop}

\begin{rmk}
Naturally, the analogous statement of Theorem \ref{thm41} for $N_k(A_2/(x^m))$ can be carried out without any significant change. The only difference is that we need to replace Proposition \ref{prop44} by Kerchev's result \cite{kerchev2013}. If we combine this reasoning with the computation in Theorem \ref{thm37}, we actually prove:
\begin{prop}
As a $W_1$-module, the Jordan-H\"older series of $N_2(A_2/(x^m))$ is \[N_2(A_2/(x^m))\sim F_1[1]+F_1[2]+\dots+F_1[m-1].\]
\end{prop}
\end{rmk}

\section{Finite-Dimensionality}

Now, as promised, let us make a comparison between the phenomena we observed in previous sections. In both cases, the associated ``commutative'' spaces, $\mathrm{Spec~}\qq[x,y]/(x^m+y^m)$ and $\mathrm{Spec~}\qq[x,y]/(x^m)$, are one-dimensional. However, as we have seen, $B_k(A_2/(x^m+y^m))$ is a finite-dimensional $\qq$-vector space while $B_k(A_2/(x^m))$ is of infinite dimension. This difference is encoded in the information about the locus of non-reduced points in $\textrm{Spec~}A_{\textrm{ab}}$, the spectrum associated with the abelianization of the algebra we began with. The exact statement is the following theorem:

\begin{thm}\label{thm51}
Let $A$ be a finitely generated graded associative algebra over $\qq$ (generated by elements of degree 1) such that $X:=A_{\textrm{ab}}$ is at most of dimension 1 in the Krull sense. Moreover, if we assume that the number of non-reduced points in $\textrm{Spec~}X$ is finite, then $B_k(A)$ and $N_j(A)$ are finite-dimensional $\qq$-vector spaces for $k\geq3$ and $j\geq 2$.
\end{thm}

\begin{rmk}
Part of this theorem was stated in Jordan and Orem's paper \cite{jordan2014} as Corollary 3.10. We obtained this result independently and generalized it to $B_k$.
\end{rmk}

To prove the theorem, we need to cite the following powerful lemma:

\begin{lemma}(\cite{bapat2013})\label{lemma52}
$[M_j,L_k]\subset L_{k+j}$, whenever $j$ is odd.
\end{lemma}

Let us prove Theorem \ref{thm51}.

\begin{proof}
Apply Lemma \ref{lemma52} to the case $k=1$ and $j=2r+1$, then $[L_1,M_{2r+1}]\subset L_{2r+2}$, which implies \begin{equation}\label{eq1}\sum_i[x_i,M_{2r+1}]\subset L_{2r+2},\end{equation} where $x_i$ are degree-one generators of $A$. On the other hand, we always have $L_j\subset\sum_i[x_i,M_{j-1}]$ by definition. In particular, set $j=2r$, we get \begin{equation}\label{eq2}L_{2r}\subset\sum_i[x_i,M_{2r-1}].\end{equation} Combining (\ref{eq1}) and (\ref{eq2}), we get (to be understood in each graded component) \begin{equation}\label{eq3}\dim L_{2r}-\dim L_{2r+2}\leq\dim(\sum_i[x_i,M_{2r-1}])-\dim(\sum_i[x_i,M_{2r+1}]).\end{equation}

Now, let $V$ be the $\qq$-vector space spanned by the $x_i$'s, which we assume to be $n$-dimensional. Let \[\begin{split}\lambda:V\otimes M_{2r-1}&\to\sum_i[x_i,M_{2r-1}]\\ \mu:V\otimes M_{2r+1}&\to\sum_i[x_i,M_{2r+1}]\end{split}\] be the natural maps sending $a\otimes b$ to $[a,b]$. We have the following commutative diagram: \[\xymatrix{0\ar[r] & \ker \mu\ar[r]\ar[d]_\alpha & V\otimes M_{2r+1}\ar[r]^\mu\ar[d]_\beta & \sum_i[x_i,M_{2r+1}]\ar[r]\ar[d]_\gamma & 0\\ 0\ar[r] & \ker \lambda\ar[r] & V\otimes M_{2r-1}\ar[r]^\lambda & \sum_i[x_i,M_{2r-1}]\ar[r] & 0}\] where both rows are exact, and $\beta$ and $\gamma$ are injections. The snake lemma gives a long exact sequence that reduces to a short exact sequence \[0\to\coker\alpha\to\coker\beta\to\coker\gamma\to0,\] hence (again, to be understood in each graded component) \[\dim\coker\gamma=\dim\coker\beta-\dim\coker\alpha\leq\dim\coker\beta.\] That is, \[\begin{split}\dim(\sum_i[x_i,M_{2r-1}])-\dim(\sum_i[x_i,M_{2r+1}])&\leq\dim V\otimes M_{2r-1}-\dim V\otimes M_{2r+1}\\ &=n(\dim M_{2r-1}-\dim M_{2r+1}).\end{split}\] Plugging it in (\ref{eq3}), we conclude \begin{equation}\label{eq4}\dim B_{2r}+\dim B_{2r+1}\leq n(\dim N_{2r-1}+\dim N_{2r}).\end{equation}
A similar argument gives \begin{equation}\label{eq5}\dim B_{2r+1}\leq n\dim N_{2r}.\end{equation}
Therefore we only have to prove the finite-dimensionality of $N_j$ for $j\geq2$.

It is known that $C:=A/M_3(A)$ has a commutative ring structure: the multiplication is defined by $a*b=(ab+ba)/2$. In addition, each $N_j$ has a $C$-module structure as follows: for $a\in C$ and $m\in N_j$, we also use $a\in A$ and $m\in M_j$ to denote the fixed liftings of $a$ and $m$. The multiplication is again defined by $a*m=(am+ma)/2$. Note $m\in M_j$, $(am+ma)/2\in M_j$, we define $a*m\in N_j$ to be the equivalent class of $(am+ma)/2$.

It is straightforward to check that this multiplication is well-defined. We should point out that $M_jM_3\subset M_{j+2}$ (\cite{bapat2013} Corollary 1.4) is needed in this routine verification.

It is also established (see \cite{etingof2009}, \cite{jordan2014}) that $C$ is a finitely generated algebra over $\qq$ and $N_j$ a finitely generated module over $C$. By a standard result from commutative algebra, in order to show that $N_j$ is a finite dimensional $\qq$-vector space, we only have to show that $N_j$ has finite support as a coherent sheaf on $\mathrm{Spec~}C$.

Now let $X:=A_{\textrm{ab}}=A/M_2(A)$ and consider the short exact sequence \[0\to M_2(A)/M_3(A)\to C\to X\to 0.\] Using results from \cite{jennings1947}, we find that $M_2(A)/M_3(A)$ is a nilpotent ideal in $C$. In other words, $\mathrm{Spec~}X$ and $\mathrm{Spec~}C$ share the same reduced scheme (in particular, the same underlying topological space).

The goal is to prove that $N_j$ has finite support, and by our assumption, we need to deal with reduced points only. Further, since $\textrm{Spec~}X$ is assumed to be at most one-dimensional, it suffices to prove that $N_j$ is supported on the set of singular points. Let $x$ be a reduced and smooth point of $\textrm{Spec~}X$. Then, from \cite{jordan2014} Corollary 3.9, we have $N_j(A)_x=N_j(A_x)$, where the subscript $x$ denotes the completion at $x$. Moreover, our assumptions on $X$ guarantee that $A_x$ is commutative (being $\qq[[x]]$ or $\qq$), therefore $N_j(A)_x=0$. As a consequence, $N_j$ has finite support and it is of finite dimension.

\end{proof}

\section{Acknowledgement}

This paper originates from the RSI projects suggested to the authors by Prof.~Pavel Etingof in Summer 2012. The authors are very grateful for the ideas that Pavel kindly shared with us, as well as for the inspiring discussions we enjoyed with him. The email correspondence with Krasilnikov was very helpful. We also would like to express our thanks to Martina Balagovic and David Jordan for their help on Magma codes.

\section{Appendix: Torsion Subgroups in $N_k(A_n(\zz))$}

It was noticed in \cite{bhupatiraju2012} that there are no torsion elements in $N_k(A_n(\zz))$ when $k$ and $n$ are small. A natural guess would be that this holds for every $k$ and $n$. Surprisingly, Krasilnikov found a 3-torsion in $N_3(A_5(\zz))$ and more torsion elements for higher $k$ and $n$ (see \cite{krasilnikov2013}). He predicts that 3-torsion is the only possible one showing up in $N_k(A_n(\zz))$, and they arise somehow ``by coincidence''.

In fact, Krasilnikov and his coauthors obtained many more results (see \cite{deryabina2013}, \cite{dacosta2013}). For instance, in \cite{deryabina2013} they proved that \[\{x_1^{n_1}\cdots x_k^{n_k}[x_{i_1},[x_{i_2},x_{i_3}]]\cdot[x_{i_4},x_{i_5}]\cdots[x_{i_{2l}},x_{i_{2l+1}}]\}_{i_1<i_2<\dots<i_{2l+1}}\] forms a basis of 3-torsion in $N_3(A_n)$. In particular, the number of copies of $\zz/3$ in $N_3(A_n)[1,1,\dots,1]$ is given by the sum of binomial coefficients \[{n\choose5}+{n\choose7}+{n\choose9}+\dots\] Our data obtained by Magma agrees with Krasilnikov's result. Similarly, many results about $N_4(A_n)$ can be deduced from \cite{dacosta2013}, in which the generators of $M_5(A_n)$ were determined explicitly.

This appendix serves as an experimental verification of Krasilnikov's predictions. The computational data obtained by Magma is listed in the next page. For example, if we assign each variable to be of degree 1, then $N_5(A_5(\zz))$ has torsion subgroup $\zz/3$ in degree $(1,1,1,1,3)$ and $(\zz/3)^2$ in degree $(1,1,1,2,3)$. A blank entry indicates that there is no torsion elements. Red question marks ``\red{?}'' represent data currently beyond our computability. A particular interesting phenomenon is that no torsion is found in $N_4$ so far, therefore we do not include the $N_4$ row in the table.

\newpage
\setlength{\textheight}{220mm}

\begin{landscape}
$N_k(A_2)$: No torsion up to total degree 14.\\
~\\

$N_k(A_3)$: No torsion up to total degree 10.\\
~\\

$N_k(A_4)$: No torsion up to total degree 9.\\
~\\

\begin{tabular}{c|cccccccccccc}
$N_k(A_5)$ & \!(1,1,1,1,1)\! & \!(1,1,1,1,2)\! & \!(1,1,1,1,3)\! & \!(1,1,1,2,2)\! & \!(1,1,1,1,4)\! & \!(1,1,1,2,3)\! & \!(1,1,2,2,2)\! & \!(1,1,1,1,5)\! & \!(1,1,1,2,4)\! & \!(1,1,1,3,3)\! & \!(1,1,2,2,3)\! & \!(1,2,2,2,2)\!\\
\hline
$N_3$ & $\zz/3$ & $\zz/3$ & $\zz/3$ & $\zz/3$ & $\zz/3$ & $\zz/3$ & $\zz/3$ & $\zz/3$ & $\zz/3$ & $\zz/3$ & $\zz/3$ & $\zz/3$\\
$N_5$ & & & $\zz/3$ & $\zz/3$ & $\zz/3$ & $(\zz/3)^2$ & $(\zz/3)^3$ & $\zz/3$ & $(\zz/3)^2$ & $(\zz/3)^3$ & $(\zz/3)^4$ & \red{?}\\
$N_6$ & & & & & $\zz/3$ & & & $\zz/3$ & $\zz/3$ & & & \red{?}\\
$N_7$ & & & & & & & & $\zz/3$ & $\zz/3$ & $(\zz/3)^2$ & $(\zz/3)^2$ & \red{?}\\
\end{tabular}
~\\
~\\

\begin{tabular}{c|ccccc}
$N_k(A_6)$ & (1,1,1,1,1,1) & (1,1,1,1,1,2) & (1,1,1,1,1,3) & (1,1,1,1,2,2) & (1,1,1,1,1,4)\\
\hline
$N_3$ & $(\zz/3)^6$ & $(\zz/3)^6$ & $(\zz/3)^6$ & $(\zz/3)^6$ & $(\zz/3)^6$\\
$N_5$ & & $(\zz/3)^6$ & $(\zz/3)^{12}$ & $(\zz/3)^{18}$ & $(\zz/3)^{12}$\\
$N_6$ & & & $\zz/3$ & & $(\zz/3)^6$\\
$N_7$ & & & & & $(\zz/3)^7$\\
\end{tabular}
~\\
~\\
~\\

\begin{tabular}{c|cc}
$N_k(A_7)$ & (1,1,1,1,1,1,1) & (1,1,1,1,1,1,2)\\
\hline
$N_3$ & $(\zz/3)^{22}$ & $(\zz/3)^{22}$\\
$N_5$ & $(\zz/3)^{21}$ & $(\zz/3)^{69}$\\
$N_6$ & & $\zz/3$
\end{tabular}
\end{landscape}
\newpage

\setlength{\textheight}{210mm}

\bibliographystyle{alpha}

\bibliography{C:/Users/Piojo/Dropbox/Documents/Source}

\end{document}